\documentclass[reqno,b5paper]{amsart}
\usepackage{amsmath}
\usepackage{amssymb}
\usepackage{amsthm}
\usepackage{enumerate}
\usepackage[mathscr]{eucal}
\theoremstyle{plain}
\newtheorem{thm}{Theorem}[section]

\newtheorem{lem}[thm]{Lemma}

\theoremstyle{definition}
\newtheorem{defn}{Definition}[section]

\begin{document}

\setcounter {page}{1}
%---------------Title,Author,Abstract-----------------------------------------------
\title{$I$-completeness in Function Spaces}

\author[A.K. Banerjee, A. Banerjee]{ Amar Kumar Banerjee*, Apurba Banerjee\ }
\newcommand{\acr}{\newline\indent}
\maketitle
\address{Department of Mathematics, The University of Burdwan, Burdwan-713104, West Bengal, India.\\
                
           {*\,} Corresponding author\\
					Email: akbanerjee@math.buruniv.ac.in, akbanerjee1971@gmail.com \acr(A.K.Banerjee),\\             apurbabanerjeemath@gmail.com\acr(A.Banerjee) \\}

\maketitle
\begin{abstract}
In this paper we have studied the idea of ideal completeness of function spaces $Y^{X}$ with respect to pointwise uniformity and uniformity of uniform convergence. Further involving topological structure on $X$ we have obtained relationships between the uniformity of uniform convergence on compacta on $Y^{X}$ and uniformity of uniform convergence on $Y^{X}$ in terms of $I$-Cauchy condition and $I$-convergence of a net. Also using the notion of a $k$-space we have given a sufficient condition for $C(X,Y)$ to be ideal complete with respect to the uniformity of uniform convergence on compacta.
\end{abstract}
\author{}
\maketitle


%-------------------------Section 1- Background and introduction-----------------------
\section{\textbf{Introduction}}
The idea of convergence of a real sequence was extended to statistical convergence by H. Fast (\cite{FH}) (see also I.J. Schoenberg \cite{SIJ} ) as follows:   \\
If $\mathbb{N}$ denotes the set of all natural numbers and $K\subset \mathbb{N}$ then $K_{n}$ denotes the set $\{k\in K: k\leq n\}$ and $\left|K_{n}\right|$ stands for the cardinality of the set $K_{n}$. The natural density of the set $K$ is defined by $d(K)=$ lim$_{n\rightarrow \infty}\frac{\left|K_{n}\right|}{n}$, provided the limit exists.  \\
A sequence $\{x_{n}\}$ of points in the real number space is said to be statistically convergent to $x_{0}$ if for arbitrary $\epsilon>0$ the set $K(\epsilon)=\{k\in \mathbb{N}: \left|x_{k} - x_{0}\right|\geq \epsilon\}$ has natural density zero. Lot of works have been done so far on such convergence and its topological consequences after the initial works by T. Salat (\cite{ST}). However if one considers the concept of nets instead of sequences the above approach does not seem to be appropriate because of the absence of any idea of density in an arbitrary directed set. Instead it seems more appropriate to follow the more general approach of ideal convergence \cite{KP}. In \cite{KP} (see also \cite{KPMM}) a generalization of the notion of statistical convergence was proposed as follows:
A subcollection $I\subset 2^{\mathbb{N}}$ is called an ideal if (i) $A,B\in I$ implies $A\cup B\in I$ and (ii) $A\in I$, $B\subset A$ imply $B\in I$. $I$ is called non-trivial ideal if $I\neq \{\Phi\}$ and $\mathbb{N}\notin I$. $I$ is called admissible if it contains all singletons. If $I$ is a proper non-trivial ideal then the family of sets $F(I)=\{M\subset \mathbb{N}: \mathbb{N}\setminus M\in I\}$ is a filter on $\mathbb{N}$ and it is called the filter associated with the ideal $I$ of $\mathbb{N}$. It is easy to check that the family $I_{d}=\{A\subset \mathbb{N}: d(A)=0\}$ forms a non-trivial admissible ideal of $\mathbb{N}$.    \\
A sequence $\{x_{n}\}$ of real numbers is said to be $I$-convergent to $x_{0}\in X$ (in short $x_{0}=I$-lim$_{n\rightarrow \infty}x_{n}$) if $K(\epsilon)\in I$ for each $\epsilon>0$, where $K(\epsilon)=\{k\in \mathbb{N}: \left|x_{k} - x_{0}\right|\geq \epsilon\}$.  \\
Lahiri and Das (\cite{LD2}) extended the idea of $I$-convergence to an arbitrary topological space and observed that few basic properties related to ideal convergence are also preserved like ordinary convergence in a topological space. They also introduced in \cite{LD3} the idea of $I$-convergence of nets in a topological space and examined how far it affects the basic properties. Later Das and Ghosal (\cite{DG}) introduced the idea of $I$-Cauchy nets in a uniform space and formulated two equivalent forms of $I$-Cauchy condition of a net in a uniform space. Also they proved that every $I$-convergent net in a uniform space with respect to the uniform topology satisfies $I$-Cauchy condition. Further they have given a sufficient condition for uniform spaces to be complete in terms of $I$-convergence of $I$-Cauchy nets.  \\
In this paper we have studied the idea of ideal completeness of a uniform space and have shown a sufficient condition for a subspace of product uniform space with respect to the pointwise uniformity to be ideal complete. Again we have given a necessary and sufficient condition for a product uniform space with respect to the uniformity of uniform convergence to be ideal complete. Also we have obtained that if a uniform space $(Y,\mathcal{U})$ is ideal complete then so are $(Y^{X},\mathcal{U}_{u})$ and $(C(X,Y),\mathcal{U}_{u})$, where $X$ is a non-empty set and $(Y^{X},\mathcal{U}_{u})$ is the product uniform space with respect to the uniformity of uniform convergence $\mathcal{U}_{u}$ and $C(X,Y)$ is the space of all continuous functions from $X$ to $Y$.   \\ 
Further involving  topology on $X$ in the function space we have shown separately that a net $\{f_{\lambda}:\lambda\in \Lambda\}$ in $(Y^{X},\mathcal{U}_{k})$ where $\mathcal{U}_{k}$ is the uniformity of uniform convergence on compacta is an $I$-Cauchy net if and only if for each compact subset $K$ of $X$, $\{f_{\lambda}|_{K}:\lambda\in \Lambda\}$ is $I$-Cauchy in the uniformity of uniform convergence on $K$ and the same result has been given in case of $I$-convergence of a net in $(Y^{X},\tau_{\mathcal{U}_{k}})$ where $\tau_{\mathcal{U}_{k}}$ is the topology of uniform convergence on compacta. Finally  applying the idea of a $k$-space and using the preceding results we have shown that if $X$ is a $k$-space and $(Y,\mathcal{U})$ is ideal complete, then $C(X,Y)$ is ideal complete in the uniformity of uniform convergence on compacta.

\section{\textbf{Ideal completeness in function spaces}}
Let $(Y,\tau)$ be a topological space and $X$ be a non-empty set. Let $Y^{X}$ be endowed with the Tychonoff product topology. We say a subcollection $\mathcal{F}\subset Y^{X}$ has the topology of pointwise convergence (or, the pointwise topology) iff it is provided with the subspace topology induced by the Tychonoff product topology on $Y^{X}$.
Let $(Y,\mathcal{U})$ be a uniform space which we will write sometimes simply as $Y$. It may be recalled that  for any point $x$ in a uniform space $(Y,\mathcal{U})$, the collection $\{U[x]: x\in Y\}$ [where $U[x]= \{y\in Y: (x,y)\in U\}$] forms a local neighbourhood base at $x$. The corresponding topology $\tau_{\mathcal{U}}$ is called the uniform topology on $Y$. By an open set in $Y$ we shall always mean an open set in the uniform topology in $Y$.   \\
Let $(D,\geq)$ be a directed set and $I$ be a non-trivial ideal of $D$. A net in $Y$ will be denoted by $\{s_{\alpha}: \alpha\in D\}$ or simply by $\{s_{\alpha}\}$, when there is no confusion about $D$. For $\alpha \in D$ let $M_{\alpha}= \{\beta\in D: \beta\geq \alpha\}$. Then the collection $F_{0}=\{A\subset D: A\supset M_{\alpha},$ for some $\alpha\in D\}$ forms a filter in $D$. Let $I_{0}=\{B\subset D: D\setminus B\in F_{0}\}$. Then $I_{0}$ is a non-trivial ideal of $D$.
%---------------------------Definition 2.1------------------------------------------
\begin{defn}
(\cite{LD3}) A non-trivial ideal $I$ of $D$ will be called $D$-admissible if $M_{\alpha}\in F(I)$ for all $\alpha\in D$.
\end{defn}
% --------------------------Definition 2.2-------------------------------------------
\begin{defn}
(\cite{LD3}) A net $\{s_{\alpha}: \alpha\in D\}$ in $(Y,\mathcal{U})$ is said to be $I$-convergent to $x_{0}\in Y$ if for any open set $U$ in $(Y,\tau_{\mathcal{U}})$ containing $x_{0}$, $\{\alpha\in D: s_{\alpha}\notin U\}\in I$.
\end{defn}
%-----------------------------Definition 2.3------------------------------------------
\begin{defn}
(\cite{DG}) A net $\{s_{\alpha}: \alpha\in D\}$ in a uniform space $(Y,\mathcal{U})$ is said to be $I$-Cauchy if for any $U\in \mathcal{U}$, there exists a $\beta\in D$ such that $\{\alpha\in D: (s_{\alpha},s_{\beta})\notin U\}\in I$.
\end{defn}  
It is easy to check that when $I=I_{0}$ the definition of $I$-Cauchy condition of a net coincides with the usual Cauchy condition.   \\

We know two equivalent forms of $I$-Cauchy condition of a net in a uniform space which are stated below.        \\

% ---------------------------------Theorem 2.1------------------------------------
\begin{thm} \cite{DG}
 For a net $\{s_{\alpha}:\alpha\in D\}$ in a uniform space $(Y,\mathcal{U})$ the following conditions are equivqlent:    \\
 $\mathbf{(i)}$  $\{s_{\alpha}:\alpha\in D\}$ is an $I$-Cauchy net.   \\
 $\mathbf{(ii)}$ For every $U\in \mathcal{U}$ there exists $A\in I$ such that $\alpha,\beta\notin A$ implies $(s_{\alpha},s_{\beta})\in U$. \\
 $\mathbf{(iii)}$ For every $U\in \mathcal{U}$, $\{\beta\in D:E_{\beta}(U)\notin I\}\in I$, where $E_{\beta}(U)=\{\alpha\in D:(s_{\alpha},s_{\beta})\notin U\}$.
 
\end{thm}
Throughout we assume that $\Lambda$ is a directed set, $(Y,\tau)$ is a topological space and $X$ is a non-empty set, $Y^{X}$ is endowed with the Tychonoff product topology and  $\mathcal{F}\subset Y^{X}$ has the pointwise topology (i.e., subspace topology induced by the Tychonoff product topology on $Y^{X}$) unless otherwise stated.
%-----------------------------------Theorem 2.2--------------------------------------
\begin{thm} \cite{BB}
If $\mathcal{F}$ has the pointwise topology, then a net $\{f_{\lambda}: \lambda\in \Lambda\}$ is $I$-convergent to $f$ in $\mathcal{F}$ if and only if the net $\{f_{\lambda}(x): \lambda\in \Lambda\}$ is $I$-convergent to $f(x)$ in $\pi_{x}(\mathcal{F})$ for each $x\in X$ where $I$ is a non-trivial ideal of the directed set $\Lambda$ and $\pi_{x}$ is the $x$-th projection map from $Y^{X}$ onto $Y$.
\end{thm}
\begin{proof}
Since $\pi_{x}(f_{\lambda})=f_{\lambda}(x)$ for $x\in X$, the proof follows from Theorem 3.4(\cite{BB}).
\end{proof}
We turn now to the discussion of defining a uniformity on the product of uniform spaces, subject to the obvious restriction that the topology of such a uniformity should be the product topology.   \\
First we recall the following definition.
% ----------------------------Definition2.4-------------------------------------------
\begin{defn} (\cite{SW})
If $X_{\alpha}$ is a set for each $\alpha\in \mathcal{A}$ and $X=\prod{X_{\alpha}}$, the $\alpha$th biprojection is the map $P_{\alpha}: X\times X\rightarrow X_{\alpha}\times X_{\alpha}$ defined by $P_{\alpha}(x,y)=(\pi_{\alpha}(x),\pi_{\alpha}(y))$, where $\pi_{\alpha}$ is the $\alpha$th projection mapping from $X$ onto $X_{\alpha}$.
\end{defn}
% ---------------------------------Theorem 2.3------------------------------------
\begin{thm} \cite{SW}
If $\mathcal{U}_{\alpha}$ is a diagonal uniformity on $X_{\alpha}$, for each $\alpha\in \mathcal{A}$, then the sets $P_{\alpha_{1}}^{-1}(U_{\alpha_{1}})\cap P_{\alpha_{2}}^{-1}(U_{\alpha_{2}})\cap\cdots \cap P_{\alpha_{n}}^{-1}(U_{\alpha_{n}})$, where $U_{\alpha_{i}}\in \mathcal{U}_{\alpha_{i}}$, for $i=1,2,\ldots,n$, form a base for a uniformity $\mathcal{U}_{p}$ on $\prod{X_{\alpha}}$ which is called the product uniformity on $\prod{X_{\alpha}}$ and whose associated topology is the product topology on $\prod{X_{\alpha}}$.  
\end{thm}
Now let us assume that $(Y,\mathcal{U})$ is a uniform space.
%---------------------------------Definition 2.5--------------------------------
\begin{defn}
(\cite{SW}) The product uniformity $\mathcal{U}_{p}$ on $Y^{X}$ is called the uniformity of pointwise convergence or the pointwise uniformity.
\end{defn}
Note that the topology associated with the pointwise uniformity on $Y^{X}$ is, of course, the pointwise topology.    \\
 
%-----------------------------------Theorem 2.4--------------------------------------
\begin{thm}
$\{f_{\lambda}: \lambda\in \Lambda\}$ is an $I$-Cauchy net in $Y^{X}$ with the pointwise uniformity if and only if $\{f_{\lambda}(x): \lambda\in \Lambda\}$ is an $I$-Cauchy net in $Y$ for each $x\in X$, where $I$ is a non-trivial ideal of $\Lambda$. 
\end{thm}
\begin{proof}
Let $\{f_{\lambda}: \lambda\in \Lambda\}$ be an $I$-Cauchy net in $Y^{X}$ with the pointwise uniformity. Then for every member of $\mathcal{U}_{p}$ of the form $P_{x}^{-1}(U)$ where $U\in \mathcal{U}$ there exists $B\in I$ such that $\alpha,\beta\notin B$ implies $(f_{\alpha},f_{\beta})\in P_{x}^{-1}(U)$, i.e., $P_{x}(f_{\alpha},f_{\beta})\in U$, i.e., $(f_{\alpha}(x),f_{\beta}(x))\in U$ and hence by Theorem 2.1 it follows that $\{f_{\alpha}(x): \alpha\in \Lambda\}$ is an $I$-Cauchy net in $Y$ for each $x\in X$.  \\
Conversely suppose $\{f_{\lambda}(x): \lambda\in \Lambda\}$ is an $I$-Cauchy net in $Y$ for each $x\in X$. Hence by Theorem 2.1 for each $U\in \mathcal{U}$ and for each $x\in X$ there exists $A_{0}\in I$ such that $\alpha,\beta\notin A_{0}$ implies $(f_{\alpha}(x),f_{\beta}(x))\in U$. Let us choose a member $U\in \mathcal{U}_{p}$. Then $U$ has the form $P_{x_{1}}^{-1}(U_{1})\cap P_{x_{2}}^{-1}(U_{2})\cap\cdots \cap P_{x_{n}}^{-1}(U_{n})$ where $U_{i}\in \mathcal{U}$ for all $i=1,2,....,n$ and $x_{1},x_{2},....,x_{n}\in X$. Now for $U_{1},U_{2},\ldots,U_{n}\in \mathcal{U}$ and $x_{1},x_{2},\ldots,x_{n}\in X$ there exist $A_{1},A_{2},\ldots,A_{n}\in I$ such that $\lambda_{1},\lambda_{2}\notin A_{i}$ implies $(f_{\lambda_{1}}(x_{i}),f_{\lambda_{2}}(x_{i}))\in U_{i}$ for $i=1,2,\ldots,n$. Let $A=\bigcup_{i=1}^{n}A_{i}$. Then $A\in I$. Now $\alpha,\beta\notin A$ implies $(f_{\alpha}(x_{i}),f_{\beta}(x_{i}))\in U_{i}$ for each $i=1,2,\ldots,n$, which in turn implies that $P_{x_{i}}(f_{\alpha},f_{\beta})\in U_{i}$ for each $i=1,2,\ldots,n$ i.e., $(f_{\alpha},f_{\beta})\in P_{x_{i}}^{-1}(U_{i})$ for each $i=1,2,\ldots,n$. Thus $\alpha,\beta\notin A$ implies $(f_{\alpha},f_{\beta})\in \bigcap_{i=1}^{n}P_{x_{i}}^{-1}(U_{i})=U\in \mathcal{U}_{p}$. Hence again applying Theorem 2.1 we get that $\{f_{\lambda}: \lambda\in \Lambda\}$ is an $I$-Cauchy net in $Y^{X}$.
\end{proof} 
We define below the notion of ideal completeness of a uniform space in the same manner as that of a uniform space to be complete.
% ----------------------------Definition2.6-------------------------------------------
\begin{defn}
A uniform space $(Y,\mathcal{U})$ is said to be ideal complete if every net $\{s_{\alpha}: \alpha\in \Lambda\}$ in $(Y,\mathcal{U})$ which is $I$-Cauchy in $(Y,\mathcal{U})$, is $I$-convergent in $(Y,\tau_{\mathcal{U}})$, where $I$ is a non-trivial ideal of $\Lambda$ and $\tau_{\mathcal{U}}$ is the uniform topology on $Y$ corresponding to the uniformity $\mathcal{U}$ on $Y$. 
\end{defn}
%-----------------------------------Theorem 2.5--------------------------------------
\begin{thm}
Let $\mathcal{F}\subset Y^{X}$ be a function space with the pointwise uniformity. Let $\Lambda$ be a directed set and $I$ be a non-trivial ideal of $\Lambda$. Then $\mathcal{F}$ is ideal complete if   \\
$\mathbf{(a)}$ $\mathcal{F}$ is pointwise closed (i.e., $\mathcal{F}$ is closed in the pointwise topology on $Y^{X}$),   \\
$\mathbf{(b)}$ $\pi_{x}(\mathcal{F})=\{f(x):f\in \mathcal{F}\}$ is ideal complete in $Y$ for each $x\in X$. 
\end{thm}
\begin{proof}
Let $\{f_{\lambda}: \lambda\in \Lambda\}$ be an $I$-Cauchy net in $\mathcal{F}$. Then for each $x\in X$, $\{\pi_{x}(f_{\lambda}): \lambda\in \Lambda\}$ is an $I$-Cauchy net in $\pi_{x}(\mathcal{F})$. Since $\pi_{x}(\mathcal{F})$ is ideal complete, so by definition $\{\pi_{x}(f_{\lambda}): \lambda\in \Lambda\}$ is $I$-convergent to some $f(x)\in \pi_{x}(\mathcal{F})$ and this holds for each $x\in X$. Thus we see that $\{\pi_{x}(f_{\lambda}): \lambda\in \Lambda\}$ is $I$-convergent to $f(x)$ in $\pi_{x}(\mathcal{F})$ for each $x\in X$. Now applying Theorem 2.2 we get that $\{f_{\lambda}: \lambda\in \Lambda\}$ is $I$-convergent to $f$ in $Y^{X}$. Since $\mathcal{F}$ is pointwise closed so we have $f\in \mathcal{F}$. Hence the result follows.
\end{proof}
We know that pointwise $I$-limit of continuous functions (on the real line, say) need not be continuous, so that $C(X,Y)$, the space of all continuous functions from $X$ to $Y$ is not always ideal complete in the uniformity of pointwise convergence.    \\
The uniformity of pointwise convergence and its topology occupy one end of the spectrum of structures used to make function spaces out of collections of functions. At the other end sit the uniformity of uniform convergence and its topology which we recapitulate below.
% ----------------------------Definition2.7-------------------------------------------
\begin{defn} (\cite{SW})
If $Y$ has a uniformity $\mathcal{U}$, the family of sets of the form 
\begin{center}
$E_{U}=\{(f,g)\in Y^{X}\times Y^{X}: (f(x),g(x))\in U$ for each $x\in X\}$
\end{center}
for $U\in \mathcal{U}$, form a base  for a uniformity $\mathcal{U}_{u}$ on $Y^{X}$ called the uniformity of uniform convergence. Its associated topology, $\tau_{\mathcal{U}_{u}}$, is the topology of uniform convergence.   
\end{defn}
%-----------------------------Definition 2.8-------------------------------------------
\begin{defn}
If a net $\{f_{\lambda}:\lambda\in \Lambda\}$ in $Y^{X}$is $I$-convergent to $f\in Y^{X}$ in the topology of uniform convergence, we say $\{f_{\lambda}: \lambda\in \Lambda\}$ is uniformly $I$-convergent to $f$, where $I$ is a non-trivial ideal of $\Lambda$. 
\end{defn}
%-----------------------------Definition 2.9------------------------------------------
\begin{defn}
If a net $\{f_{\lambda}:\lambda\in \Lambda\}$ in $Y^{X}$ is $I$-Cauchy in the uniformity of uniform convergence then we call $\{f_{\lambda}:\lambda\in \Lambda\}$ uniformly $I$-Cauchy, where $I$ is a non-trivial ideal of $\Lambda$.
\end{defn}
The next theorem provides a relationship between pointwise $I$-convergence and uniform $I$-convergence of a net and subsequently gives a necessary and sufficient condition for a product uniform space with respect to the uniformity of uniform convergence to be ideal complete.
%-----------------------------------Theorem 2.6--------------------------------------
\begin{thm}
A net $\{f_{\lambda}:\lambda\in \Lambda\}$ in $(Y^{X},\tau_{\mathcal{U}_{u}})$ is uniformly $I$-convergent to $f$ if and only if $\mathbf{(i)}$ the net $\{f_{\lambda}:\lambda\in \Lambda\}$ is uniformly $I$-Cauchy in $(Y^{X},\mathcal{U}_{u})$ and $\mathbf{(ii)}$ the net $\{f_{\lambda}(x):\lambda\in \Lambda\}$ is $I$-convergent to $f(x)$ in $(Y,\tau_{\mathcal{U}})$ for each $x\in X$           \\
 $\mathbf{[}$ i.e., the net $\{f_{\lambda}:\lambda\in \Lambda\}$ is pointwise $I$-convergent to $f$ in $(Y^{X},\tau_{\mathcal{U}_{u}})$ $\mathbf{]}$ where $I$ is a non-trivial ideal of $\Lambda$.
\end{thm}
\begin{proof}
Let $\{f_{\lambda}:\lambda\in \Lambda\}$ be a net in $(Y^{X},\mathcal{U}_{u})$ which is uniformly $I$-convergent to $f\in Y^{X}$. So by Theorem 2 of (\cite{DG}) it follows that $\{f_{\lambda}:\lambda\in \Lambda\}$ is uniformly $I$-Cauchy. Now by definition of $I$-convergence of a net for any $U\in \mathcal{U}$ we have the set $\{\lambda\in \Lambda: f_{\lambda}\notin E_{U}[f]\}\in I$, where $E_{U}[f]=\{g\in Y^{X}: (f,g)\in E_{U}\}=\{g\in Y^{X}: (f(x),g(x))\in U$ for each $x\in X\}$. This implies the set $\{\lambda\in \Lambda: f_{\lambda}\in E_{U}[f]\}\in F(I)$, where $F(I)$ is the filter associated with the ideal $I$. Let $A=\{\lambda\in \Lambda: f_{\lambda}\in E_{U}[f]\}$. Then $\lambda_{0}\in A$ implies $f_{\lambda_{0}}\in E_{U}[f]$, i.e., $(f,f_{\lambda_{0}})\in E_{U}$ and so $(f(x),f_{\lambda_{0}}(x))\in U$ for each $x\in X$, i.e., $f_{\lambda_{0}}(x)\in U[f(x)]$ for each $x\in X$. Now let $x_{0}\in X$ be an arbitrary element. Then if $A_{0}=\{\alpha\in \Lambda: f_{\alpha}(x_{0})\in U[f(x_{0})]\}$, we see that $A\subset A_{0}$ and since $A\in F(I)$ so we have by definition of a filter $A_{0}\in F(I)$ as well. So, the set $\{\alpha\in \Lambda: f_{\alpha}(x_{0})\notin U[f(x_{0})]\}\in I$ i.e., the net $\{f_{\lambda}(x_{0}): \lambda\in \Lambda\}$ is $I$-convergent to $f(x_{0})$ in $(Y,\tau_{\mathcal{U}})$. Hence we conclude that the net $\{f_{\lambda}(x):\lambda\in \Lambda\}$ is $I$-convergent to $f(x)$ in $(Y,\tau_{\mathcal{U}})$ for each $x\in X$.  \\

Conversely, let the net $\{f_{\lambda}:\lambda\in \Lambda\}$ in $Y^{X}$ be pointwise $I$-convergent to $f$ i.e., the net $\{f_{\lambda}(x):\lambda\in \Lambda\}$ is $I$-convergent to $f(x)$ in $(Y,\tau_{\mathcal{U}})$ for each $x\in X$ and $\{f_{\lambda}:\lambda\in \Lambda\}$ is uniformly $I$-Cauchy in $(Y^{X},\mathcal{U}_{u})$. To show that $\{f_{\lambda}:\lambda\in \Lambda\}$ is uniformly $I$-convergent to $f$ in $(Y^{X},\tau_{\mathcal{U}_{u}})$, we are to show that for any $U\in \mathcal{U}$ the set $\{\lambda\in \Lambda: f_{\lambda}\notin E_{U}[f]\}\in I$.  Now for each $x\in X$ we have the set $\{\lambda\in \Lambda: f_{\lambda}(x)\notin U[f(x)]\}\in I$, i.e., $\{\lambda\in \Lambda: f_{\lambda}(x)\in U[f(x)]\}\in F(I)$, where $F(I)$ is the filter associated with the ideal $I$. For each $x\in X$ let us call the set $B_{x}=\{\lambda\in \Lambda: f_{\lambda}(x)\in U[f(x)]\}$. Then $B_{x}\in F(I)$ for each $x\in X$. Now choose a symmetric $V\in \mathcal{U}$ such that $V\circ V\subset U$. For each $x\in X$ let us call $C_{x}=\{\lambda\in \Lambda: f_{\lambda}(x)\in V[f(x)]\}$. Then again on the basis of the condition \textbf{(ii)} we have $C_{x}\in F(I)$ for each $x\in X$. Since $\{f_{\lambda}:\lambda\in \Lambda\}$ is uniformly $I$-Cauchy in $(Y^{X},\mathcal{U}_{u})$ so by Theorem 2.1 we have for $V\in \mathcal{U}$ there exists $A\in I$ such that $\alpha, \beta\notin A$ implies $(f_{\alpha},f_{\beta})\in E_{V}$. We will prove that for any arbitrary $X_{0}\in X$, $A^{c}\subset B_{x_{0}}$. Now, since $A^{c},C_{x_{0}}\in F(I)$ so $A^{c}\cap C_{x_{0}}\neq \phi$. Let us take $\alpha_{0}\in A^{c}\cap C_{x_{0}}$. Then for any $\alpha\in A^{c}$ we have $(f_{\alpha_{0}},f_{\alpha})\in E_{V}$, i.e., $(f_{\alpha_{0}}(x),f_{\alpha}(x))\in V$ for each $x\in X$. Hence in particular $(f_{\alpha_{0}}(x_{0}),f_{\alpha}(x_{0}))\in V$. Again since $\alpha_{0}\in C_{x_{0}}$, so, $f_{\alpha_{0}}(x_{0})\in V[f(x_{0})]$, i.e., $(f(x_{0}),f_{\alpha_{0}}(x_{0}))\in V$. Thus we get $(f(x_{0}),f_{\alpha}(x_{0}))\in V\circ V\subset U$, i.e., $f_{\alpha}(x_{0})\in U[f(x_{0})]$, i.e., $\alpha\in B_{x_{0}}$. Hence $A^{c}\subset B_{x_{0}}$. Since $x_{0}\in X$ has been chosen arbitrarily so we conclude that $A^{c}\subset \bigcap_{x\in X}B_{x}$. This in turn implies that $\bigcap_{x\in X}B_{x}\in F(I)$. Now we see that $\bigcap_{x\in X}B_{x}=\{\lambda\in \Lambda: f_{\lambda}(x)\in U[f(x)]$ for each $x\in X\}= \{\lambda\in \Lambda: f_{\lambda}\in E_{U}[f]\}$. Thus in turn we have proved that $\{\lambda\in \Lambda: f_{\lambda}\notin E_{U}[f]\}\in I$. Hence the result follows.
\end{proof}
%-----------------------------------Theorem 2.7--------------------------------------
\begin{thm}
If a uniform space $(Y,\mathcal{U})$ is ideal complete then so are        \\
$\mathbf{(a)}$ $(Y^{X},\mathcal{U}_{u})$         \\
$\mathbf{(b)}$ $(C(X,Y),\mathcal{U}_{u})$, where $C(X,Y)$ is the space of all continuous functions from $X$ to $Y$ with the uniformity of uniform convergence $\mathcal{U}_{u}$.
\end{thm}
\begin{proof}
$\mathbf{(a)}$ Let a net $\{f_{\lambda}:\lambda\in \Lambda\}$ be uniformly $I$-Cauchy in $(Y^{X},\mathcal{U}_{u})$. Then the net $\{f_{\lambda}(x):\lambda\in \Lambda\}$ is $I$-Cauchy in $(Y,\mathcal{U})$ for each $x\in X$. Hence $\{f_{\lambda}(x):\lambda\in \Lambda\}$ is $I$-convergent to some $f(x)\in Y$, since $(Y,\mathcal{U})$ is ideal complete space. By previous result the function $f\in Y^{X}$ defined by $f(x)=I$-lim $f_{\lambda}(x)$ for each $x\in X$ is uniform $I$-limit of the net $\{f_{\lambda}:\lambda\in \Lambda\}$. Thus $(Y^{X},\mathcal{U}_{u})$ is ideal complete.    \\
$\mathbf{(b)}$ It has been proved in Theorem 42.10 of (\cite{SW}) that $(C(X,Y),\mathcal{U}_{u})$ is a closed subspace of $(Y^{X},\mathcal{U}_{u})$.  \\
First we prove that if $(M,\mathcal{D})$ is a uniform space and $N\subset M$ then an $I$-Cauchy net $\{x_{\lambda}: \lambda\in \Lambda\}$ in $(N,\mathcal{D}_{N})$ is also $I$-Cauchy in $(M,\mathcal{D})$, where $I$ is a non-trivial ideal of $\Lambda$ and $\mathcal{D}_{N}$ is the relative uniformity induced on $N$ by $\mathcal{D}$. Now since $\{x_{\lambda}: \lambda\in \Lambda\}$ is an $I$-Cauchy net in $(N,\mathcal{D}_{N})$ so for each $D_{N}\in \mathcal{D}_{N}$ there exists some $\lambda_{0}\in \Lambda$ such that the set $\{\lambda\in \Lambda: (x_{\lambda},x_{\lambda_{0}})\notin D_{N}\}\in I$. Now $D_{N}=D\cap (N\times N)$ where $D\in \mathcal{D}$ and for each $D\in \mathcal{D}$ there corresponds a $D_{N}\in \mathcal{D}_{N}$. Also we note that $\{\lambda\in \Lambda: (x_{\lambda},x_{\lambda_{0}})\notin D\}\subset \{\lambda\in \Lambda: (x_{\lambda},x_{\lambda_{0}})\notin D_{N}\}$.  Hence $\{\lambda\in \Lambda: (x_{\lambda},x_{\lambda_{0}})\notin D\}\in I$. So, for each $D\in \mathcal{D}$ there is some $\lambda_{0}\in \Lambda$ such that the set $\{\lambda\in \Lambda: (x_{\lambda},x_{\lambda_{0}})\notin D\}\in I$. Hence $\{x_{\lambda}: \lambda\in \Lambda\}$ becomes an $I$-Cauchy net in $(M,\mathcal{D})$.   \\
Secondly, we show  that if $(M,\mathcal{D})$ is an ideal complete uniform space and $N$ be a closed subset of $(M,\tau_{\mathcal{D}})$ then $(N,\mathcal{D}_{N})$ becomes an ideal complete space.      \\
Since an $I$-Cauchy net $\{x_{\lambda}: \lambda\in \Lambda\}$ in $(N,\mathcal{D}_{N})$ is also $I$-Cauchy in $(M,\mathcal{D})$ and if $(M,\mathcal{D})$ happens to be ideal complete space so the net $\{x_{\lambda}: \lambda\in \Lambda\}$ is $I$-convergent to some $x_{0}\in M$ in $(M,\tau_{\mathcal{D}})$. Now if $x_{0}\in N$ then $\{x_{\lambda}: \lambda\in \Lambda\}$ is $I$-convergent in $(N,\tau_{\mathcal{D}_{N}})$. But if $x_{0}\notin N$ then $\{x_{\lambda}: \lambda\in \Lambda\}$ is a net in $N\setminus \{x_{0}\}$ such that it is $I$-convergent to $x_{0}\in M$. Then $x_{0}$ becomes a limit point of $N$ [by Theorem 3 of (\cite{LD3})]. Since $N$ is closed in $(M,\tau_{\mathcal{D}})$ then $x_{0}\in N$. In any case $x_{0}\in N$ if $N$ is closed in $(M,\tau_{\mathcal{D}})$. Hence $\{x_{\lambda}: \lambda\in \Lambda\}$ is an $I$-Cauchy net in $(N,\mathcal{D}_{N})$ which becomes $I$-convergent in $(N,\tau_{\mathcal{D}_{N}})$. \\
Thus we conclude that $(C(X,Y),\mathcal{U}_{u})$ is an ideal complete subspace of $(Y^{X},\mathcal{U}_{u})$.
 \end{proof}
If we involve the topology of $X$ in our function space and $Y$ has a uniform structure we can have a uniform structure on $Y^{X}$ which is called the uniformity of uniform convergence on compacta or the uniformity of compact convergence. We recall below the definition of that uniformity and the associated topology.
%--------------------------------Definition 2.10---------------------------
\begin{defn} (\cite{SW})
Suppose $Y$ has a uniformity $\mathcal{U}$. The uniformity of uniform convergence on compacta or the uniformity of compact convergence, $\mathcal{U}_{k}$, has for a subbase the sets  
\begin{center}
$E_{K,U}=\{(f,g)\in Y^{X}\times Y^{X}: (f(x),g(x))\in U$, for each $x\in K\}$
\end{center}
where $K$ is a compact subset of $X$ and $U\in \mathcal{U}$. The topology $\tau_{\mathcal{U}_k}$ thus induced on $Y^{X}$ is the topology of compact convergence.
\end{defn}
%------------------------------------Theorem 2.8------------------------
\begin{thm}
A net $\{f_{\lambda}:\lambda\in \Lambda\}$ is $I$-convergent to $f$ in $(Y^{X},\tau_{\mathcal{U}_{k}})$ where $\tau_{\mathcal{U}_{k}}$ is the topology of uniform convergence on compacta if and only if for each compact subset $K$ of $X$, $\{f_{\lambda}|_{K}:\lambda\in \Lambda\}$ is uniformly $I$-convergent to $f|_{K}$ in $(Y^{K},\tau_{\mathcal{U}_{u}})$ where $\tau_{\mathcal{U}_{u}}$ is the topology of uniform convergence on $Y^{K}$ and $I$ is a non-trivial ideal of $\Lambda$. 
\end{thm}
\begin{proof}
Let $\{f_{\lambda}:\lambda\in \Lambda\}$ be $I$-convergent to $f$ in $(Y^{X},\tau_{\mathcal{U}_{k}})$. This implies for each subbasic open set $E_{K,U}[f]$ containing $f$ in $(Y^{X},\tau_{\mathcal{U}_{k}})$ the set $\{\lambda\in \Lambda: f_{\lambda}\notin E_{K,U}[f]\}\in I$. This holds for each fixed compact subset $K$ of $X$ and for each $U\in \mathcal{U}$. Hence for each compact subset $K$ of $X$, $\{f_{\lambda}|_{K}:\lambda\in \Lambda\}$ is uniformly $I$-convergent to $f|_{K}$ in $(Y^{K},\tau_{\mathcal{U}_{u}})$.  \\
Conversely suppose $\{f_{\lambda}|_{K}:\lambda\in \Lambda\}$ is uniformly $I$-convergent to $f|_{K}$ in $(Y^{K},\tau_{\mathcal{U}_{u}})$ for each compact subset $K$ of $X$. Let $E_{K_{i},U_{i}}[f|_{K_{i}}]$ be arbitrarily chosen basic open sets containing $f|_{K_{i}}$ in $(Y^{K_{i}},\tau_{\mathcal{U}_{u}})$ for $i=1,2,....,n$ respectively. Then we have the sets $\{\lambda\in \Lambda: f_{\lambda}|_{K_{i}}\notin E_{K_{i},U_{i}}[f|_{K_{i}}]\}\in I$ for all $i=1,2,....,n$. This implies the sets $\{\lambda\in \Lambda: f_{\lambda}|_{K_{i}}\in E_{K_{i},U_{i}}[f|_{K_{i}}]\}\in F(I)$ for all $i=1,2,....,n$, where $F(I)$ is the filter associated with the ideal $I$. We note below the follwing observation.   \\
Let $K$ be any compact subset of $X$, $U$ be any member of the uniformity $\mathcal{U}$ on $Y$ and $f\in Y^{X}$. Now we see that if $E_{K,U}[f|_{K}]$ be a basic open set (as per definition 2.7) containing $f|_{K}$ in $(Y^{K},\tau_{\mathcal{U}_{u}})$ and $E_{K,U}[f]$ be a subbasic open set (as per difinition 2.10) containing $f$ in $(Y^{X},\tau_{\mathcal{U}_{k}})$ then 
\begin{center}
 $E_{K,U}[f|_{K}]= \{h\in Y^{K}: (f|_{K},h)\in E_{K,U}\}$ \\
$=\{h\in Y^{K}: (f|_{K}(x),h(x))\in U$ for each $x\in K\}$  \\
\end{center}
and 
\begin{center}
$E_{K,U}[f]=\{g\in Y^{X}: (f,g)\in E_{K,U}\}$  \\
$=\{g\in Y^{X}: (f(x),g(x))\in U$ for each $x\in X \}$  \\
$=\{g\in Y^{X}: (f|_{K}(x),g|_{K}(x))\in U$ for each $x\in K\}$
\end{center}
Now let $g\in Y^{X}$ be arbitrary. Then $g|_{K}\in Y^{K}$. Now it is clear from above that $g|_{K}\in E_{K,U}[f|_{K}]$ implies $g\in E_{K,U}[f]$ and conversely $g\in E_{K,U}[f]$ implies $g|_{K}\in E_{K,U}[f|_{K}]$.
Thus we obtain $\{\lambda\in \Lambda: f_{\lambda}|_{K}\in E_{K,U}[f|_{K}]\}=\{\lambda\in \Lambda: f_{\lambda}\in E_{K,U}[f]\}$. Consequently we get from the above observation that $\{\lambda\in \Lambda: f_{\lambda}\in E_{K_{i},U_{i}}[f]\}\in F(I)$ for all $i=1,2,....,n$.
Hence $\bigcap_{i=1}^{n}\{\lambda\in \Lambda: f_{\lambda}\in E_{K_{i},U_{i}}[f]\}\in F(I)$, i.e., $\{\lambda\in \Lambda: f_{\lambda}\in \bigcap_{i=1}^{n} (E_{K_{i},U_{i}}[f])\}\in F(I)$, i.e., $\{\lambda\in \Lambda: f_{\lambda}\in (\bigcap_{i=1}^{n} E_{K_{i},U_{i}})[f]\}\in F(I)$ . So we get, $\{\lambda\in \Lambda: f_{\lambda}\notin (\bigcap_{i=1}^{n}E_{K_{i},U_{i}})[f]\}\in I$. Now keeping in mind that ($\bigcap_{i=1}^{n}E_{K_{i},U_{i}})[f]$ being a basic open set in $(Y^{X},\tau_{\mathcal{U}_{k}})$ containing $f$ the net $\{f_{\lambda}:\lambda\in \Lambda\}$ is $I$-convergent to $f$ in $(Y^{X},\tau_{\mathcal{U}_{k}})$.
\end{proof}
%----------------------------Theorem 2.9-----------------------------
\begin{thm}
A net $\{f_{\lambda}:\lambda\in \Lambda\}$ is an $I$-Cauchy net in $(Y^{X},\mathcal{U}_{k})$ where $\mathcal{U}_{k}$ is the uniformity of uniform convergence on compacta if and only if for each compact subset $K$ of $X$, $\{f_{\lambda}|_{K}:\lambda\in \Lambda\}$ is uniformly $I$-Cauchy in $(Y^{K},\mathcal{U}_{u})$ where $\mathcal{U}_{u}$ is the uniformity of uniform convergence on $Y^{K}$ and $I$ is a non-trivial ideal of $\Lambda$. 
\end{thm}
\begin{proof}
Let $\{f_{\lambda}:\lambda\in \Lambda\}$ be an $I$-Cauchy net in $(Y^{X},\mathcal{U}_{k})$. Then for each subbasic element $E_{K,U}$ there exists some $\lambda_{0}\in \Lambda$ such that the set $\{\lambda\in \Lambda: (f_{\lambda},f_{\lambda_{0}})\notin E_{K,U}\}\in I$. This holds for each fixed compact subset $K$ of $X$ and for each $U\in \mathcal{U}$. Hence $\{f_{\lambda}|_{K}:\lambda\in \Lambda\}$ is uniformly $I$-Cauchy in $(Y^{K},\mathcal{U}_{u})$ for each compact subset $K$ of $X$.     \\
Conversely let $\{f_{\lambda}|_{K}:\lambda\in \Lambda\}$ be uniformly $I$-Cauchy in $(Y^{K},\mathcal{U}_{u})$ for each compact subset $K$ of $X$. Then for arbitrarily chosen basic elements $E_{K_{1},U_{1}}, E_{K_{2},U_{2}},\\ 
 \ldots, E_{K_{n},U_{n}}$ in $(Y^{K_{1}},\mathcal{U}_{u})$, $(Y^{K_{2}},\mathcal{U}_{u})$,\ldots, $(Y^{K_{n}},\mathcal{U}_{u})$ respectively we have by Theorem 2.1(\cite{DG}) there exist $A_{1},A_{2},.....,A_{n}\in I$ such that $\alpha,\beta\notin A_{i}$ implies $(f_{\alpha}|_{K_{i}},f_{\beta}|_{K_{i}})\in E_{K_{i},U_{i}}$ for each $i=1,2,\ldots,n$. We note below the following observation.  \\
Let $K$ be any compact subset of $X$, $U$ be any member of the uniformity $\mathcal{U}$ on $Y$. Now we see that if $E_{K,U}$ be a basic element for the uniformity $\mathcal{U}_{u}$ on $Y^{K}$ and a subbasic element for the uniformity $\mathcal{U}_{k}$ on $Y^{X}$, then for the first case mentioned $E_{K,U}=\{(f,g)\in Y^{k}\times Y^{K}:(f(x),g(x))\in U$ for each $x\in K\}\subset \{(f,g)\in Y^{X}\times Y^{X}:(f(x),g(x))\in U$ for each $x\in K\}=E_{K,U}$ mentioned for the second case.  \\
 Thus we obtain $(f_{\alpha}|_{K},f_{\beta}|_{K})\in E_{K,U}$ implies $(f_{\alpha},f_{\beta})\in E_{K,U}$. Let us say $A=\bigcup_{i=1}^{n}A_{i}$. Then clearly $A\in I$ and it follows from the observation made just before that $\alpha,\beta\notin A$ implies $(f_{\alpha},f_{\beta})\in E_{K_{i},U_{i}}$ for all $i=1,2,....,n$. Hence we can conclude that $\alpha,\beta\notin A$ implies $(f_{\alpha},f_{\beta})\in \bigcap_{i=1}^{n}E_{K_{i},U_{i}}$. Since $\bigcap_{i=1}^{n}E_{K_{i},U_{i}}$ is a basic element for the uniformity $\mathcal{U}_{k}$ on $Y^{X}$, so, it follows by Theorem 2.1(\cite{DG}) that $\{f_{\lambda}:\lambda\in \Lambda\}$ is an $I$-Cauchy net in $(Y^{X},\mathcal{U}_{k})$.
\end{proof}
We now recall the concept of a topological space to be a $k$-space which plays a central role in the discussion of both completeness and compactness relative to the uniformity of uniform convergence on compacta and its topology.
%--------------------------------------Definition 2.11----------------------------------
\begin{defn}(\cite{SW}) 
A topological space $(X,\tau)$ is a $k$-space (or a compactly generated space) iff the following condition holds:         \\
(a) $A\subset X$ is open in $(X,\tau)$ iff $A\cap K$ is open in $(K,\tau_{K})$ for each compact set $K$ in $(X,\tau)$.
\end{defn}
The $k$-spaces are important to our discussion of $I$-convergence of continuous functions on compacta because, in these spaces, the continuous functions are precisely those which behave well on compact subsets. The proof of the following lemma, which says this more precisely, follows easily in applying the definition of a $k$-space. 
%--------------------------------------Lemma 2.1-------------------------
\begin{lem} \cite{SW}
If $X$ is a $k$-space and $Y$ is a topological space, then $f:X\rightarrow Y$ is continuous iff $f|_{K}$ is continuous for each compact $K\subset X$.  
\end{lem}
Using this result and theorems 2.8 and 2.9, which describes $I$-convergence on compacta as being precisely uniform $I$-convergence on each compact subset, the following theorem holds good.   \\
%-------------------------------------Theorem 2.10-----------------------------------
\begin{thm}
If $X$ is a $k$-space and $(Y,\mathcal{U})$ is an ideal complete uniform space, then $C(X,Y)$ is ideal complete in the uniformity of uniform convergence on compacta. 
\end{thm} 
\begin{proof}
 At first from Theorem 2.9 we know that if $\{f_{\lambda}:\lambda\in \Lambda\}$ is an $I$-Cauchy net in $(Y^{X},\mathcal{U}_{k})$ where $\mathcal{U}_{k}$ is the uniformity of uniform convergence on compacta then we have $\{f_{\lambda}|_{K}:\lambda\in \Lambda\}$ is uniformly $I$-Cauchy in $(Y^{K},\mathcal{U}_{u})$ for each compact subset $K\subset X$, where $\mathcal{U}_{u}$ is the uniformity of uniform convergence on $Y^{K}$ and $I$ is a non-trivial ideal of $\Lambda$. On the other hand from Theorem 2.7 we know that for each compact $K\subset X$, $C(K,Y)$ is an ideal complete subspace of $Y^{K}$ in the uniformity of uniform convergence. Now let $\{f_{\lambda}:\lambda\in \Lambda\}$ be an $I$-Cauchy net in $(C(X,Y),\mathcal{U}_{k})$. Then $\{f_{\lambda}|_{K}:\lambda\in \Lambda\}$ is uniformly $I$-Cauchy in $(C(K,Y),\mathcal{U}_{u})$ for each compact $K\subset X$. Since $(C(K,Y),\mathcal{U}_{u})$ is ideal complete, so, a continuous uniform $I$-limit $f_{K}:K\rightarrow Y$ exists for each compact $K\subset X$. It can be seen easily that if $K_{1}\subset K_{2}\subset X$ then $f_{K_{2}}|_{K_{1}}=f_{K_{1}}$, and from this it follows that the function $f:X\rightarrow Y$ defined by $f(x)=f_{K}(x)$ for $x\in K$, is well defined. It is continuous by above lemma 2.1, and since $\{f_{\lambda}|_{K}:\lambda\in \Lambda\}$ is uniformly $I$-convergent to $f|_{K}$ in $(Y^{K},\tau_{\mathcal{U}_{u}})$ for each compact $K\subset X$, so, by Theorem 2.8 it follows that $\{f_{\lambda}:\lambda\in \Lambda\}$ is $I$-Convergent to $f$ in $(Y^{X},\mathcal{U}_{k})$. Hence we get that $\{f_{\lambda}:\lambda\in \Lambda\}$ is $I$-Convergent to $f$ in $(C(X,Y),\mathcal{U}_{k})$. Thus $(C(X,Y),\mathcal{U}_{k})$ is ideal complete.
\end{proof}
\noindent\textbf{Acknowledgement:} The second author is thankful to The University Grants Commission, Government of India for giving the award of Senior Research Fellowsip during the tenure of preparation of this research paper.

\end{document}